\documentclass[12pt, reqno, twoside, letterpaper]{amsart}

\usepackage{palindromes_Waring}

\title[Every natural number is the sum of forty-nine palindromes]
{Every natural number is the sum\break of forty-nine palindromes}
      
\author[W.\ D.\ Banks]{William D.\ Banks}

\address{Department of Mathematics, 
         University of Missouri, 
         Columbia MO, USA.}

\email{bankswd@missouri.edu}
        
\date{\today}

\begin{document}

\begin{abstract}
It is shown that the set of decimal palindromes is an additive basis for the
natural numbers. Specifically, we prove that every natural number can be expressed
as the sum of forty-nine (possibly zero) decimal palindromes.
\end{abstract}

\maketitle


\vskip1.8in

\section{Statement of result}

Let $\NN\defeq\{0,1,2,\ldots\}$ denote the set of natural numbers (including zero).
Every number $n\in\NN$ has a unique decimal representation of the form
\begin{equation}
\label{eq:representation}
n=\sum_{j=0}^{L-1}10^j\delta_j,
\end{equation}
where each digit $\delta_j$ belongs to the digit set
$$
\cD\defeq\{0,1,2,\ldots,9\},
$$
and the leading digit $\delta_{L-1}$ is nonzero whenever $L\ge 2$.
In what follows, we use diagrams to illustrate the ideas; for example,
$$
n=\begin{tabular}{|c|c|c|c|}
\hline
$\delta_{L-1}$&$\cdots$&$\delta_1$&$\delta_0$\\
\hline
\end{tabular}
$$
represents the relation \eqref{eq:representation}. The integer $n$ is said to be a
\emph{palindrome} if its digits satisfy the symmetry condition
$$
\delta_j=\delta_{L-1-j}\qquad(0\le j<L).
$$
Denoting by $\cP$ the collection of all palindromes in $\NN$,
the aim of this note is to show that $\cP$ is an additive basis for $\NN$.

\begin{theorem}
\label{thm:main}
The set $\cP$ of decimal palindromes is an additive basis for the
natural numbers $\NN$.
Every natural number is the sum of forty-nine (possibly zero) decimal palindromes.
\end{theorem}

The proof is given in the next section.  It is unlikely that
the second statement is optimal; a refinement of our method may yield an
improvement.  No attempt has been made to generalize this theorem to bases
other than ten; for large bases, this should be straightforward, but small
bases may present new obstacles (for example, obtaining the correct analogue of
Lemma~\ref{lem:passage} may be challenging in the binary case, where
the only nonzero digit is the digit one).
We remark that arithmetic properties of palindromes (in various bases) have been previously
investigated by many authors; see\cite{AllSha,BanHarSak,BanShp1,BanShp2,BerZie,CilLucShp,CilTesLuc,Col,
CooGab,Goi,HerLuc,Kor,Luc,Sim} and the references therein.  

\section{The proof}
\label{sec:proof}

\subsection{Notation}
\label{sec:notate}

For every $n\in\NN$, let $L(n)$ (the ``length'' of $n$)
denote the number of decimal digits $L$ in the expansion
\eqref{eq:representation}; in particular, $L(0)\defeq 1$.

For any $\ell\in\NN$ and $d\in\cD$, we denote
\begin{equation}
\label{eq:plddefn}
p_\ell(d)\defeq\begin{cases}
0&\quad\hbox{if $\ell=0$};\\
d&\quad\hbox{if $\ell=1$};\\
10^{\ell-1}d+d&\quad\hbox{if $\ell\ge 2$}.
\end{cases}
\end{equation}
Note that $p_\ell(d)$ is a palindrome, and $L(p_\ell(d))=\ell$ if $d\ne 0$.
If $\ell\ge 2$, then the decimal expansion of $p_\ell(d)$ has the form
$$
p_\ell(d)=\begin{tabular}{|c|c|c|}
\hline
$d$&$0\cdots 0$&$d$\\
\hline
\end{tabular}
$$
with $\ell-2$ zeros nested between two copies of the digit $d$.

More generally, for any integers $\ell\ge k\ge 0$ and $d\in\cD$, let
$$
p_{\ell,k}(d)\defeq 10^k p_{\ell-k}(d)=\begin{cases}
0&\quad\hbox{if $\ell=k$};\\
10^k d&\quad\hbox{if $\ell=k+1$};\\
10^{\ell-1}d+10^k d&\quad\hbox{if $\ell\ge k+2$}.
\end{cases}
$$
If $\ell\ge k+2$, then the
decimal expansion of $p_{\ell,k}(d)$ has the form
$$
p_{\ell,k}(d)=\begin{tabular}{|c|c|c|c|}
\hline
$d$&$0\cdots 0$&$d$&$0\cdots 0$\\
\hline
\end{tabular}
$$
with $\ell-k-2$ zeros nested between two copies of the digit $d$, followed
by $k$ copies of the digit zero.

Next, for any integers $\ell,k\in\NN$, $\ell\ge k+4$, and digits
$a,b\in\cD$, we denote
\begin{equation}
\label{eq:qlkabdefn}
q_{\ell,k}(a,b)\defeq p_{\ell,k}(a)+p_{\ell-1,k}(b)=
10^{\ell-1}a+10^{\ell-2}b+10^k(a+b).
\end{equation}
Taking into account that the relation $n=10\cdot\fl{n/10}+\delta_0(n)$
holds for every natural number
$n<100$, where $\fl{\cdot}$ is the floor function and $\delta_0(r)$ denotes
the one's digit of any natural number $r$, one sees that the
decimal expansion of $q_{\ell,k}(a,b)$ has the form
$$
q_{\ell,k}(a,b)=\begin{tabular}{|c|c|c|c|c|c|}
\hline
$a$&$b$&$0\cdots 0$&$\fl{(a+b)/10}$&$\delta_0(a+b)$&$0\cdots 0$\\
\hline
\end{tabular}
$$
with $\ell-k-4$ zeros nested between the digits $a,b$ and the digits of $a+b$,
followed by $k$ copies of the digit zero.  For example,
$q_{10,2}(7,8)=7800001500$.

Finally, for any integers $\ell,k\in\NN$, $\ell\ge k+4$, and digits $a,b,c\in\cD$,
$a\ne 0$, we denote by $\NN_{\ell,k}(a,b;c)$ the set of natural numbers
described as follows.  Given $n\in\NN$, let $L$ and $\{\delta_j\}_{j=0}^{L-1}$
be defined as in \eqref{eq:representation}.  Then
$\NN_{\ell,k}(a,b;c)$ consists of those integers $n$ for which
$L=\ell$, $\delta_{\ell-1}=a$, $\delta_{\ell-2}=b$, $\delta_k=c$, and $10^k\mid n$.
In other words, $\NN_{\ell,k}(a,b;c)$ is the set of natural numbers $n$ that
have a decimal expansion of the form
$$
n=\begin{tabular}{|c|c|c|c|c|}
\hline
$a$&$b$&$\bigast\cdots\bigast$&$c$&$0\cdots 0$\\
\hline
\end{tabular}
$$
with $\ell-k-3$ arbitrary digits nested between the digits $a,b$ and the digit $c$,
followed by $k$ copies of the digit zero.  We reiterate that $a\ne 0$.

\subsection{Handling small integers}
\label{sec:handling}

Let $f:\cD\to\cD$ be the function whose values
are provided by the following table:
$$
\begin{tabular}{|c||c|c|c|c|c|c|c|c|c|c|}
\hline
$d$&$0$&$1$&$2$&$3$&$4$&$5$&$6$&$7$&$8$&$9$\\
\hline
$f(d)$&$0$&$1$&$9$&$8$&$7$&$6$&$5$&$4$&$3$&$2$\\
\hline
\end{tabular}
$$
We begin our proof of Theorem~\ref{thm:main} with the following observation.

\begin{lemma}
\label{lem:10^jd-f(d)}
Every number $f(d)$ is a palindrome, and $10^jd-f(d)$ is a palindrome
for every integer $j\ge 1$.
\end{lemma}

\begin{proof}
This is easily seen if $d=0$ or $j=1$.  For $j\ge 2$
and $d=1$, the number $10^jd-f(d)=10^j-1$ is a repunit of the form $9\cdots 9$,
hence a palindrome.
Finally, for $j\ge 2$ and $2\le d\le 9$, the number $10^jd-f(d)$ is a palindrome
that has a decimal expansion of the form
$$
10^jd-f(d)=\begin{tabular}{|c|c|c|}
\hline
$d-1$&$9\cdots 9$&$d-1$\\
\hline
\end{tabular}
$$
with $j-2$ nines nested between two copies of the digit $d-1$.
\end{proof}

\begin{lemma}
\label{lem:small}
If $n$ is a natural number with at most $K$ nonzero decimal digits,
then $n$ is the sum of $2K+1$ palindromes.
\end{lemma}

\begin{proof} Starting with the expansion \eqref{eq:representation} we write
$$
n=\delta_0+\sum_{j\in\cJ}10^j\delta_j,
$$
where
$$
\cJ\defeq\big\{1\le j<L:\delta_j\ne 0\big\}.
$$
Since
\begin{equation}
\label{eq:nexpr}
n=\delta_0
+\sum_{j\in\cJ}f(\delta_j)
+\sum_{j\in\cJ}\bigl(10^j\delta_j-f(\delta_j)\bigr),
\end{equation}
Lemma~\ref{lem:10^jd-f(d)} implies that $n$ is
the sum of $2|\cJ|+1$ palindromes.  Since zero is
a palindrome, we obtain the stated result by adding
$2K-2|\cJ|$ additional zeros on the right side of \eqref{eq:nexpr}.
\end{proof}

Lemma~\ref{lem:small} implies, in particular, that $n\in\NN$ is a sum
of 49 palindromes whenever $L(n)\le 24$.  Therefore, we can
assume that $L(n)\ge 25$ in the sequel.

\subsection{Reduction to $\NN_{\ell,0}(5^+;c)$} Recall the definition
of $\NN_{\ell,k}(a,b;c)$ given in \S\ref{sec:notate}. For any given integers
$\ell,k\in\NN$, $\ell\ge k+4$, and a digit $c\in\cD$, we now denote
$$
\NN_{\ell,k}(5^+;c)\defeq\bigcup_{\substack{a,b\in\cD\\a\ge 5}}
\NN_{\ell,k}(a,b;c).
$$
The set $\NN_{\ell,k}(5^+;c)$ can be described as follows.
For each $n\in\NN$, let $L$ and $\{\delta_j\}_{j=0}^{L-1}$
be defined as in \eqref{eq:representation}.
The set $\NN_{\ell,k}(5^+;c)$ consists of those integers $n$ for which
$L=\ell$, $\delta_{\ell-1}\ge 5$, $\delta_k=c$, and $10^k\mid n$.
In other words, $\NN_{\ell,k}(5^+;c)$ is the set of natural numbers $n$ that have a
decimal expansion of the form
$$
n=\begin{tabular}{|c|c|c|c|c|}
\hline
$a$&$\bigast\cdots\bigast$&$c$&$0\cdots 0$\\
\hline
\end{tabular}
$$
with $\ell-k-2$ arbitrary digits nested between the digit $a$ ($\ge 5$)
and the digit $c$, followed by $k$ copies of the digit zero.

\begin{lemma}
\label{lem:1streduction}
Let $n\in\NN$, and put $L\defeq L(n)$ as in \eqref{eq:representation}.
If $L\ge 5$, then $n$ is the sum of two palindromes and an element of
$\NN_{\ell,0}(5^+;c)$ with some $\ell\in\{L-1,L\}$ and $c\in\cD$.
\end{lemma}

\begin{proof}
Let $\{\delta_j\}_{j=0}^{L-1}$ be defined as in \eqref{eq:representation}.
If the leading digit $\delta_{L-1}$ exceeds four, then
$n\in\NN_{L,0}(5^+;\delta_0)$, and there is nothing to prove (since zero
is a palindrome).

Now suppose that $\delta_{L-1}\le 4$. Put $m\defeq 10\delta_{L-1}+\delta_{L-2}-6$,
and observe that $4\le m\le 43$.  If  $4\le m\le 9$, then using
\eqref{eq:plddefn} we see that
\begin{align*}
n-p_{L-1}(m)
&=n-(10^{L-2}m+m)\\
&=\sum_{j=0}^{L-1}10^j\delta_j-10^{L-2}(10\delta_{L-1}+\delta_{L-2}-6)-m\\
&=6\cdot 10^{L-2}+\sum_{j=0}^{L-3}10^j\delta_j-m,
\end{align*}
and the latter number evidently lies in $\NN_{L-1,0}(5^+;c)$, where
$c\equiv (\delta_0-m)\bmod 10$.
Since $p_{L-1}(m)$ is a palindrome, this yields the desired result for $4\le m\le 9$.

In the case that $10\le m\le 43$, we write $m=10a+b$ with digits $a,b\in\cD$, $a\ne 0$.
Using \eqref{eq:qlkabdefn} we have
\begin{align*}
n-q_{L,0}(a,b)
&=n-(10^{L-1}a+10^{L-2}b+a+b)\\
&=n-(10^{L-2}m+a+b)\\
&=\sum_{j=0}^{L-1}10^j\delta_j-10^{L-2}(10\delta_{L-1}+\delta_{L-2}-6)-a-b\\
&=6\cdot 10^{L-2}+\sum_{j=0}^{L-3}10^j\delta_j-a-b,
\end{align*}
and the latter number lies in $\NN_{L-1,0}(5^+;c)$, where
$c\equiv (\delta_0-a-b)\bmod 10$.  Since $q_{L,0}(a,b)$ is the sum of two
palindromes, we are done in this case as well.
\end{proof}

\subsection{Inductive passage from $\NN_{\ell,k}(5^+;c_1)$ to $\NN_{\ell-1,k+1}(5^+;c_2)$}

\begin{lemma}
\label{lem:passage}
Let $\ell,k\in\NN$, $\ell\ge k+6$, and $c_\ell\in\cD$ be given.
Given $n\in\NN_{\ell,k}(5^+;c_1)$, one can find digits
$a_1,\ldots,a_{18},b_1,\ldots,b_{18}\in\cD\setminus\{0\}$
and $c_2\in\cD$ such that the number
$$
n-\sum_{j=1}^{18}q_{\ell-1,k}(a_j,b_j)
$$
lies in the set $\NN_{\ell-1,k+1}(5^+;c_2)$.
\end{lemma}

\begin{proof}
Fix $n\in\NN_{\ell,k}(5^+;c_1)$, and let $\{\delta_j\}_{j=0}^{\ell-1}$
be defined as in \eqref{eq:representation} (with $L\defeq\ell$).
Let $m$ be the three-digit integer formed by
the first three digits of $n$; that is,
$$
m\defeq 100\delta_{\ell-1}+10\delta_{\ell-2}+\delta_{\ell-3}.
$$
Clearly, $m$ is an integer in the range $500\le m\le 999$,
and we have
\begin{equation}
\label{eq:expandit}
n=\sum_{j=k}^{\ell-1}10^j\delta_j=10^{\ell-3}m+\sum_{j=k}^{\ell-4}10^j\delta_j.
\end{equation}
Let us denote
$$
\cS\defeq\{19,29,39,49,59\}.
$$
In view of the fact that
$$
9\cS\defeq
\mathop{\underbracket{\hskip3pt\cS+\cdots+\cS\hskip1pt}}\limits_{\text{nine copies}}
=\{171,181,191,\ldots,531\},
$$
it is possible to find an element $h\in 9\cS$ for which $m-80<2h\le m-60$.
With $h$ fixed, let $s_1,\ldots,s_9$ be elements of $\cS$ such that
$$
s_1+\cdots+s_9=h.
$$
Finally, let $\eps_1,\ldots,\eps_9$ be natural numbers, each equal to zero or two:
$\eps_j\in\{0,2\}$ for $j=1,\ldots,9$.  A specific choice of these numbers is given
below.

We now put
$$
t_j\defeq s_j+\eps_j\mand t_{j+9}\defeq s_j-\eps_j\qquad(j=1,\ldots,9),
$$
and let $a_1,\ldots,a_{18},b_1,\ldots,b_{18}\in\cD$ be determined
from the digits of $t_1,\ldots,t_{18}$, respectively, via the relations
$$
10a_j+b_j=t_j\qquad (j=1,\ldots,18).
$$
Since
$$
\cS+2=\{21,31,41,51,61\}\mand
\cS-2=\{17,27,37,47,57\},
$$
all of the digits $a_1,\ldots,a_{18},b_1,\ldots,b_{18}$ are \emph{nonzero}, as required.

Using \eqref{eq:qlkabdefn} we compute
\begin{align*}
\sum_{j=1}^{18}q_{\ell-1,k}(a_j,b_j)
&=\sum_{j=1}^{18}\bigl(10^{\ell-2}a_j+10^{\ell-3}b_j+10^k(a_j+b_j)\bigr)\\
&=10^{\ell-3}\sum_{j=1}^{18}t_j+10^k\sum_{j=1}^{18}(a_j+b_j)\\
&=2h\cdot 10^{\ell-3}+10^k\sum_{j=1}^{18}(a_j+b_j)
\end{align*}
since
$$
t_1+\cdots+t_{18}=2(s_1+\cdots+s_9)=2h
$$
regardless of the choice of the $\eps_j$'s.
Taking \eqref{eq:expandit} into account, we have
\begin{equation}
\label{eq:eureka}
n-\sum_{j=1}^{18}q_{\ell-1,k}(a_j,b_j)
=10^{\ell-3}(m-2h)+\sum_{j=k}^{\ell-4}10^j\delta_j
-10^k\sum_{j=1}^{18}(a_j+b_j),
\end{equation}
and since $60\le m-2h<80$ it follows that the number defined by
either side of \eqref{eq:eureka} lies in the set $\NN_{\ell-1,k}(5^+;c)$,
where $c$ is the unique digit in $\cD$ determined by the congruence
\begin{equation}
\label{eq:congone}
\delta_k-\sum_{j=1}^{18}(a_j+b_j)\equiv c\bmod 10.
\end{equation}

To complete the proof, it suffices to show that for an appropriate choice of
the $\eps_j$'s we have $c=0$, for this implies that $n\in\NN_{\ell-1,k+1}(5^+;c_2)$
for some $c_2\in\cD$.  To do this, let $g(r)$ denote the sum of the
decimal digits of any $r\in\NN$.  Then
$$
\sum_{j=1}^{18}(a_j+b_j)=\sum_{j=1}^{18}g(t_j)
=\sum_{j=1}^9g(s_j+\eps_j)+\sum_{j=1}^9g(s_j-\eps_j).
$$
For every number $s\in\cS$, one readily verifies that
$$
g(s+2)+g(s-2)=2\,g(s)-9.
$$
Therefore, \eqref{eq:congone} is equivalent to the congruence condition
$$
\delta_k-\sum_{j=1}^{18}g(s_j)+9E\equiv c\bmod 10,
$$
where $E$ is the number of integers $j\in\{1,\ldots,9\}$ such that
$\eps_j\defeq 2$.  As we can clearly choose the $\eps_j$'s so the latter congruence
is satisfied with $c=0$, the proof of the lemma is complete.
\end{proof}

\subsection{Proof of Theorem~\ref{thm:main}}
Let $n$ be an arbitrary natural number.  To show that $n$ is the sum of 49 palindromes,
we can assume that $L:=L(n)$ is at least $25$, as mentioned in \S\ref{sec:handling}.
By Lemma~\ref{lem:1streduction} we can find two palindromes ${\widetilde p}_1,{\widetilde p}_2$ such that
the number
\begin{equation}
\label{eq:n1reln}
n_1\defeq n-{\widetilde p}_1-{\widetilde p}_2
\end{equation}
belongs to $\NN_{\ell,0}(5^+;c_1)$
for some $\ell\in\{L-1,L\}$ and $c_1\in\cD$.
Since $\ell\ge 24$, by Lemma~\ref{lem:passage} we can find digits
$a_1^{(1)},\ldots,a_{18}^{(1)},b_1^{(1)},\ldots,b_{18}^{(1)}\in\cD\setminus\{0\}$
and $c_2\in\cD$ such that the number
$$
n_2\defeq n_1-\sum_{j=1}^{18}q_{\ell-1,0}\bigl(a_j^{(1)},b_j^{(1)}\bigr)
$$
lies in the set $\NN_{\ell-1,1}(5^+;c_2)$.  Similarly, using Lemma~\ref{lem:passage}
again we can find digits
$a_1^{(2)},\ldots,a_{18}^{(2)},b_1^{(2)},\ldots,b_{18}^{(2)}\in\cD\setminus\{0\}$
and $c_3\in\cD$ such that
$$
n_3\defeq n_2-\sum_{j=1}^{18}q_{\ell-2,1}\bigl(a_j^{(2)},b_j^{(2)}\bigr)
$$
belongs to the set $\NN_{\ell-2,2}(5^+;c_3)$.  Proceeding inductively
in this manner, we continue to construct the sequence $n_1,n_2,n_3,\ldots$,
where each number
\begin{equation}
\label{eq:ninduct}
n_i\defeq n_{i-1}-\sum_{j=1}^{18}q_{\ell-i+1,i-2}\bigl(a_j^{(i-1)},b_j^{(i-1)}\bigr)
\end{equation}
lies in the set $\NN_{\ell-i+1,i-1}(5^+;c_i)$. The method works
until we reach a specific value of $i$, say $i\defeq\nu$, where $\ell-\nu+1<(\nu-1)+6$;
at this point, Lemma~\ref{lem:passage} can no longer be applied.

Notice that, since $\ell-\nu+1\le (\nu-1)+5$,
every element of $\NN_{\ell-\nu+1,\nu-1}(5^+;c_\nu)$
has at most five nonzero digits.  Therefore, by Lemma~\ref{lem:small} we can find eleven
palindromes ${\widetilde p}_3,{\widetilde p}_4,\ldots,{\widetilde p}_{13}$ such that
\begin{equation}
\label{eq:ni0reln}
n_{\nu}={\widetilde p}_3+{\widetilde p}_4+\cdots+{\widetilde p}_{13}.
\end{equation}
Now, combining \eqref{eq:n1reln}, \eqref{eq:ninduct} with $i=2,3,\ldots,\nu$,
and \eqref{eq:ni0reln}, we see that
$$
n=\sum_{i=1}^{13}{\widetilde p}_j+\sum_{j=1}^{18}N_j,
$$
where
$$
N_j\defeq\sum_{i=2}^\nu q_{\ell-i+1,i-2}\bigl(a_j^{(i-1)},b_j^{(i-1)}\bigr)
\qquad(j=1,\ldots,18).
$$
To complete the proof of the theorem, it remains to verify that every integer
$N_j$ is the sum of two palindromes.  Indeed, by \eqref{eq:qlkabdefn} we have
$$
N_j=\sum_{i=2}^\nu p_{\ell-i+1,i-2}\bigl(a_j^{(i-1)}\bigr)
+\sum_{i=2}^\nu p_{\ell-i,i-2}\bigl(b_j^{(i-1)}\bigr).
$$
Considering the form of the decimal expansions, for each $j$ we see that
$$
\sum_{i=2}^\nu p_{\ell-i+1,i-2}\bigl(a_j^{(i-1)}\bigr)
=\begin{tabular}{|c|c|c|c|c|c|c|}
\hline
$\vphantom{\Big|}a_j^{(1)}$&$\cdots$&$a_j^{(\nu-1)}$&$0\cdots 0$&$a_j^{(\nu-1)}$&$\cdots$&$a_j^{(1)}$\\
\hline
\end{tabular}
$$
which is a palindrome of length $\ell-1$ (since $a_j^{(1)}\ne 0$) having
precisely $2(\nu-1)$ nonzero entries, and
$$
\sum_{i=2}^\nu p_{\ell-i,i-2}\bigl(b_j^{(i-1)}\bigr)
=\begin{tabular}{|c|c|c|c|c|c|c|}
\hline
$\vphantom{\Big|}b_j^{(1)}$&$\cdots$&$b_j^{(\nu-1)}$&$0\cdots 0$&$b_j^{(\nu-1)}$&$\cdots$&$b_j^{(1)}$\\
\hline
\end{tabular}
$$
which is a palindrome of length $\ell-2$ (since $b_j^{(1)}\ne 0$), also having
precisely $2(\nu-1)$ nonzero entries.

\end{document}